\documentclass[11pt,reqno]{amsart}
\usepackage{color}

\usepackage{amsfonts,amscd}
\usepackage{amssymb}
\usepackage{url}
\usepackage[english]{babel}

\theoremstyle{plain}
\newtheorem{theorem}                 {Theorem}      [section]
\newtheorem{proposition}  [theorem]  {Proposition}
\newtheorem{corollary}    [theorem]  {Corollary}
\newtheorem{lemma}        [theorem]  {Lemma}

\theoremstyle{definition}
\newtheorem{example}      [theorem]  {Example}
\newtheorem{remark}       [theorem]  {Remark}
\newtheorem{definition}   [theorem]  {Definition}

\numberwithin{equation}{section}

\def \rn{{\mathbb R}}
\def \s{{\mathbb S}}

\def \cn{{\mathbb C}}

\def \H{{\mathbb H}}

\def \nab#1#2{\hbox{$\nabla$\kern -.3em\lower 1.0 ex
    \hbox{$#1$}\kern -.1 em {$#2$}}}

\def \SLR#1{\text{\bf SL}_{#1}(\rn)}
\def \SL2{\widetilde{\text{\bf SL}}_{2}(\rn)}

\def \SO#1{\text{\bf SO}(#1)}

\def \SU#1{\text{\bf SU}(#1)}

\def \Sp#1{\text{\bf Sp}(#1)}

\def \Sol{\text{\bf Sol}}
\def \Nil{\text{\bf Nil}}

\def \sol{\mathfrak{sol}}
\def \nil{\mathfrak{nil}}

\DeclareMathOperator{\Div}{div}

\numberwithin{equation}{section}

\begin{document}

\title[Biharmonic functions on the Thurston geometries]
{A Note on biharmonic functions on \\  the Thurston geometries}

\dedicatory{version 1.0012 - 12 December 2017}

\author{Sigmundur Gudmundsson}

\address{Mathematics, Faculty of Science\\ University of Lund\\
Box 118, Lund 221\\
Sweden}
\email{Sigmundur.Gudmundsson@math.lu.se}

\begin{abstract}
We construct new explicit proper biharmonic functions on the $3$-dimensional
Thurston geometries $\Sol$, $\Nil$, $\SL2$, $\H^2\times\rn$ and $\s^2\times\rn$.
\end{abstract}



\maketitle

\section{Introduction}

The biharmonic equation is a fourth order partial differential equation which arises in areas of continuum mechanics, including elasticity theory and the solution of Stokes flows. The literature on biharmonic functions is vast, but usually the domains are either surfaces or open subsets of flat Euclidean space $\rn^n$.

Recently, new explicit biharmonic functions were constructed on the classical compact simple Lie groups $\SU n$, $\SO n$ and $\Sp n$.  This gives examples on the 3-dimensional round sphere $\s^3\cong\SU 2$ and the standard hyperbolic space $\H^3$ via a general duality principle.  For this see the papers \cite{Gud-Mon-Rat-1} and \cite{Gud-13}.

The classical Riemannian manifolds $\rn^3$, $\s^3$ and $\H^3$ of constant curvature are all on Thurston's celebrated list of $3$-dimensional model geometries, see \cite{BW-book}, \cite{Sco} and \cite{Thu}.  The aim of this paper is to extend the investigation of biharmonic functions to the other members on Thurston's list i.e. $\Sol$, $\Nil$, $\SL2$, $\H^2\times\rn$ and $\s^2\times\rn$.  In all these cases we construct new explicit solutions to the corresponding fourth order biharmonic equation.

Our methods can also be used to manufacture proper $r$-harmonic solutions for $r>2$, see Definition \ref{definition-proper-r-harmonic}.  In this study we have chosen to mainly focus on the case when $r=2$ because of its physical relevance. The results are formulated such that the solutions are globally defined but clearly the same constructions hold even locally.  This is particularly important for the holomorphic functions in use.

\section{Proper $r$-harmonic functions}\label{section-r-harmonic}

Let $(M,g)$ be a smooth m-dimensional manifold equipped with a Riemannian metric $g$.  We complexify the tangent bundle $TM$ of $M$ to $T^{\cn}M$ and extend the metric $g$ to a complex-bilinear form on $T^{\cn}M$.  Then the gradient $\nabla f$ of a complex-valued function $f:(M,g)\to\cn$ is a section of $T^{\cn}M$.  In this situation, the well-known {\it linear} Laplace-Beltrami operator (alt. tension field) $\tau$ on $(M,g)$ acts locally on $f$ as follows
$$
\tau(f)=\Div (\nabla f)=\sum_{i,j=1}^m\frac{1}{\sqrt{|g|}} \frac{\partial}{\partial x_j}
\left(g^{ij}\, \sqrt{|g|}\, \frac{\partial f}{\partial x_i}\right).
$$
For two complex-valued functions $f,h:(M,g)\to\cn$ we have the following well-known relation
$$\tau(f\cdot h)=\tau(f)\cdot h+2\cdot\kappa(f,h)+f\cdot\tau(h),$$
where the {\it conformality} operator $\kappa$ is given by $\kappa(f,h)=g(\nabla f,\nabla h)$.  Locally this acts by
$$\kappa(f,h)=\sum_{i,j=1}^mg^{ij}\frac{\partial f}{\partial x_i}\frac{\partial h}{\partial x_j}.$$

\begin{definition}\label{definition-proper-r-harmonic}
For a positive integer $r$, the iterated Laplace-Beltrami operator $\tau^r$ is given by
$$\tau^{0} (f)=f\ \ \text{and}\ \ \tau^r (f)=\tau(\tau^{(r-1)}(f)).$$
We say that a complex-valued function $f:(M,g)\to\cn$ is
\begin{enumerate}
\item[(a)] {\it $r$-harmonic} if $\tau^r (f)=0$, and
\item[(b)] {\it proper $r$-harmonic} if $\tau^r (f)=0$ and $\tau^{(r-1)} (f)$ does not vanish identically.
\end{enumerate}
\end{definition}

It should be noted that the {\it harmonic} functions are exactly $r$-harmonic for $r=1$
and the {\it biharmonic} functions are the $2$-harmonic ones.
In some texts, the $r$-harmonic functions are also called {\it polyharmonic} of order $r$.

\section{The model geometry $\Sol$}
The model space $\Sol$ on Thurston's list can be seen as the $3$-dimensional solvable Lie subgroup
$$
\Sol=\{
\begin{bmatrix}
e^t & 0 & x\\
0 & e^{-t} & y\\
0 & 0 & 1
\end{bmatrix}|\ x,y,t\in\rn\}
$$
of $\SLR 3$.  The metric on $\Sol$ is determined by the orthonormal basis
$\{X,Y,T\}$ of its Lie algebra $\sol$ given by
$$
X=
\begin{bmatrix}
0 & 0 & 1\\
0 & 0 & 0\\
0 & 0 & 0
\end{bmatrix},\ \
Y=
\begin{bmatrix}
0 & 0 & 0\\
0 & 0 & 1\\
0 & 0 & 0
\end{bmatrix},\ \
T=
\begin{bmatrix}
1 & 0 & 0\\
0 & -1 & 0\\
0 & 0 & 0
\end{bmatrix}.
$$
In the global coordinates $(x,y,t)$ on $\Sol$ this takes the following well-known form
$$
ds^2=e^{2t}dx^2+e^{-2t}dy^2+dt^2.
$$
It is easily seen that the corresponding Laplace-Beltrami operator $\tau$ and the
conformality operator $\kappa$ satisfy
\begin{equation*}\label{tension-Sol}
\tau(f)=e^{-2t}\frac {\partial ^2 f}{\partial x^2}
+e^{2t}\frac {\partial ^2 f}{\partial y^2}
+\frac {\partial ^2 f}{\partial t^2},
\end{equation*}
and
$$
\kappa(f,h)=e^{-2t}\frac {\partial f}{\partial x}\frac {\partial h}{\partial x}
+e^{2t}\frac {\partial f}{\partial y}\frac {\partial h}{\partial y}
+\frac {\partial f}{\partial t}\frac {\partial h}{\partial t},
$$
respectively.

We now present two new families of globally defined complex-valued proper
biharmonic functions on the model geometry $\Sol$.

\begin{example}\label{exam:Sol-1}
For non-zero elements $a,b\in\cn^{4}$ let the complex-valued function
$f_1,f_2:\Sol\to\cn$ be defined by
$$f_1(x,y,t)=(a_1+a_2x+a_3y+a_4xy)$$
and
$$f_2(x,y,t)=(b_1+b_2x+b_3y+b_4xy).$$
Then a simple calculation shows that the tension field satisfies $\tau(f_1)=\tau(f_2)=0$,
so the functions $f_1$ and $f_2$ are harmonic.  It is also clear that that for any natural number $r$
and the conformality operator $\kappa$ we have $$\kappa(t^r,f_1)=\kappa(t^r,f_2)=0.$$
Next we define a sequence $\{F_r\}_{r=0}^\infty$ of function $F_r:\Sol\to\cn$ by
$$F_r(x,y,t)=t^{2r}\cdot f_1(x,y,t)+t^{2r+1}\cdot f_2(x,y,t).$$
Then the tension field $\tau (F_r)$ satisfies
\begin{eqnarray*}
\tau (F_r)&=&\tau(t^{2r})\cdot f_1+2\cdot\kappa(t^{2r},f_1)+t^{2r}\cdot \tau(f_1)\\
& &\qquad +\tau(t^{2r+1})\cdot f_2+2\cdot\kappa(t^{2r+1},f_2)+t^{2r+1}\cdot \tau(f_2)\\
&=&\tau(t^{2r})\cdot f_1+\tau(t^{2r+1})\cdot f_2\\
&=&2r(2r-1)\cdot t^{2r-2}\cdot f_1+2r(2r+1)\cdot t^{2r-1}\cdot f_2.
\end{eqnarray*}
Applying these calculations and the linearity of the tension field it is easy to see that
for each natural number $r$ the function $F_r:\Sol\to\cn$ is proper $r$-harmonic.
\end{example}

\begin{example}\label{exam:Sol-2}
For $c_{20},c_{21},c_{22}\in\cn$ and the function $f_{2x}:\Sol\to\cn$ given by
$$f_{2x}(x,y,t)=c_{22}x^2+c_{21}xe^{-t}+c_{20}e^{-2t}$$ it directly follows that
the condition $\tau(f_{2x})=0$ is equivalent to the following system of linear equations.
$$\begin{bmatrix}
2 & 0 & 1\\
0 & 1 & 0
\end{bmatrix}
\cdot
\begin{bmatrix}
c_{20}\\
c_{21}\\
c_{22}
\end{bmatrix}=0.$$
This shows that the function $f_{2x}:\Sol\to\cn$ is non-constant and harmonic if and only
if is of the form $$f_{2x}(x,y,t)=a_2(2x^2-e^{-2t}),$$ where $a_2\in\cn$ is non-zero.
Employing the symmetry of the tension field we easily see that the function
$f_2:\Sol\to\cn$ with
$$f_{2}(x,y,t)=a_2(\alpha +\beta y)(2x^2-e^{-2t})+b_2(\gamma +\delta x)(2y^2-e^{2t})$$
is non-constant and harmonic for non-zero elements $(a_2,b_2),(\alpha,\beta),(\gamma,\delta)\in\cn^2$.

For complex numbers $c_{30},c_{31},c_{32},c_{33}\in\cn$ and the function $f_{3x}:\Sol\to\cn$ given by
$$f_{3x}(x,y,t)=c_{33}x^3+c_{32}x^2e^{-t}+c_{31}xe^{-2t}+c_{30}e^{-3t}$$ the condition
$\tau(f_{3x})=0$ is equivalent to the following system of linear equations.
$$\begin{bmatrix}
9 & 0 & 2 & 0\\
0 & 2 & 0 & 3\\
0 & 0 & 1 & 0
\end{bmatrix}
\cdot
\begin{bmatrix}
c_{30}\\
c_{31}\\
c_{32}\\
c_{33}
\end{bmatrix}=0.$$
This shows that the function $f_{3x}:\Sol\to\cn$ is non-constant and harmonic if and only
if is of the form $$f_{3x}(x,y,t)=a_3(2x^3-3xe^{-2t}),$$ where $a_3\in\cn$ is non-zero.
Again utilising the symmetry of the tension field, we see that $f_3:\Sol\to\cn$
with
$$f_{3}(x,y,t)=a_3(\alpha +\beta y)(2x^3-3xe^{-2t})+b_3(\gamma +\delta x)(2y^3-3ye^{2t})$$
is non-constant and harmonic for non-zero elements $(a_2,b_2),(\alpha,\beta),(\gamma,\delta)\in\cn^2$.
This process can now be repeated for any natural number $n>3$. For example we get
$$f_{4x}(x,y,t)=a_4(8x^4-24x^2e^{-2t}+3e^{-4t}),$$
$$f_{5x}(x,y,t)=a_5(8x^5-40x^3e^{-2t}+15xe^{-4t}),$$
$$f_{6x}(x,y,t)=a_6(16x^6-120x^4e^{-2t}+90x^2e^{-4t}-5e^{-6t}),$$
$$f_{7x}(x,y,t)=a_7(16x^7-168x^5e^{-2t}+210x^3e^{-4t}-35xe^{-6t}).$$
\vskip .2cm

We conclude this example by letting $h=h_2\cdot h_3:\Sol\to\cn$ be the product of the functions
$h_2,h_3:\Sol\to\cn$ with
$$h_{2}(x,y,t)=a_2(2x^2-e^{-2t})+a_3(2x^3-3xe^{-2t})$$
and
$$h_{3}(x,y,t)=b_2(2y^2-e^{2t})+b_3(2y^3-3ye^{2t}).$$
Then it is easily shown that $$\tau(h_2\cdot h_3)=-8(a_2+3a_3x)(b_2+3b_3y)
\ \ \text{and}\ \ \tau^2(h_2\cdot h_3)=0.$$
This means that here we have a complex 4-dimensional family of proper
biharmonic functions globally defined on the model space $\Sol$.
\end{example}

\section{The model geometry $\Nil$}
The space $\Nil$ on Thurston's list can be presented as the $3$-dimensional nilpotent Lie subgroup
$$
\Nil=\{
\begin{bmatrix}
1 & x & t\\
0 & 1 & y\\
0 & 0 & 1
\end{bmatrix}|\ x,y,z\in\rn\}.
$$
of $\SLR 3$ equipped with its standard left-invariant Riemannian metric.
The restriction of this metric to $\Nil$ is determined by the orthonormal
basis $\{X,Y,T\}$ of its Lie algebra $\nil$ given by
$$
X=
\begin{bmatrix}
0 & 1 & 0\\
0 & 0 & 0\\
0 & 0 & 0
\end{bmatrix},\ \
Y=
\begin{bmatrix}
0 & 0 & 0\\
0 & 0 & 1\\
0 & 0 & 0
\end{bmatrix},\ \
T=
\begin{bmatrix}
0 & 0 & 1\\
0 & 0 & 0\\
0 & 0 & 0
\end{bmatrix}.
$$

It is well-known that in the global coordinates $(x,y,t)$ on $\Nil$ the left-invariant Riemannian metric satisfies
$$
ds^2=dx^2+dy^2+(dt-xdy)^2.
$$
A straight forward calculation shows that the corresponding Laplace-Beltrami operator $\tau$ is given by
\begin{equation*}\label{tension-Nil}
\tau(f)=(\frac {\partial ^2 f}{\partial x^2}+\frac {\partial ^2 f}{\partial y^2})
+2x\frac {\partial ^2 f}{\partial y\partial t}+(1+x^2)\frac {\partial ^2 f}{\partial t^2}.
\end{equation*}

We now give a new family of globally defined complex-valued proper biharmonic functions
on $\Nil$.

\begin{example}\label{exam:Nil-1}
For two holomorphic functions $h_1,h_2:\cn\to\cn$ and a non-zero element $a\in\cn^2$ we define the
complex-valued function $f_1:Nil\to\cn$ by
$$f_1(x,y,t)=h_1(x+iy)+h_2(x-iy)+a_1t+a_2xt.$$
Then it is clear that $f_1$ is non-contstant and harmonic i.e. $f_1\neq 0$ and $\tau(f_1)=0$.

For a non-zero element $b\in\cn^{12}$ we define the function $f_2:\Nil\to\cn$ with the following formula
\begin{eqnarray*}
f_2(x,y,t)&=&b_1x^2+b_2y^2+b_3yt+b_4x^3+b_5x^2y+b_6x^2t+b_7xy^2\\
& &\quad+b_8y^3+b_9x^3y+b_{10}xy^3+b_{11}y^2t+b_{12}x^3t.
\end{eqnarray*}
Then an elementary calculation gives
\begin{eqnarray*}
\tau(f_2)&=&2b_1+2b_2+2b_3x+6b_4x+2b_5y+2b_6t+2b_7x\\
& &\quad +6b_8y+6b_9xy+6b_{10}xy+2b_{11}(t+2xy)+6b_{12}xt.
\end{eqnarray*}
and
$\tau^2(f_2)=0$.
This shows that the function $f_2:\Nil\to\cn$ provides a 12-dimensional
family of proper biharmonic functions on $\Nil$.
\end{example}

\section{The model geometry $\SL2$}

The model space $\SL2$ on Thurston's list is diffeomorphic to the universal cover of the 3-dimensional
Lie group $\SLR 2$ of $2\times 2$ real traceless matrices.  It is well-known that
$\SL2$ can, as a Riemannian manifold, be modelled as $\rn^3$ equipped with the following metric
$$
ds^2=\frac 1{y^2}(dx^2+dy^2)+(dt+\frac{dx}y)^2.
$$
For this fact we refer to \cite{BW-book}.  This metric is different from the
one obtained by lifting the standard metric of $\SLR 2$ to its
universal cover.  It is also clear that it is not a product metric induced by metrics on
$\rn^2$ and $\rn$, respectively.
The Laplace-Beltrami operator on $\SL2$ with the above metric $ds^2$ satisfies
\begin{equation*}\label{tension-SL2}
\tau(f)=y^2(\frac {\partial ^2 f}{\partial x^2}+\frac {\partial ^2 f}{\partial y^2})
+2\frac {\partial ^2 f}{\partial t^2}-2y\frac {\partial ^2 f}{\partial x\partial t}.
\end{equation*}

We now present a new family of globally defined complex-valued proper biharmonic functions
on the model space $\SL2$.

\begin{example}\label{exam:SL2-1}
Let $h_1,h_2:\cn\to\cn$ be holomorphic functions on $\cn$ and for a non-zero element $a\in\cn^2$
we define the function $f_1:\SL2\to\cn$ by
$$f_1(x,y,t)=h_1(x+iy)+h_2(x-iy)+a_1t+a_2yt.$$
Then it immediately follows from the above formula for the Laplace-Beltrami
operator that $\tau(f_1)=0$ so the non-constant function $f_1$ is harmonic.

For a non-zero $b\in\cn^6$ let $f_2:\SL2\to\cn$ be the complex-valued function satisfying
$$f_2(x,y,t)=b_1xt+b_2t^2+b_3xt^2+b_4yt^2+b_5t^3+b_6yt^3.$$
Then an elementary calculation shows that
$$\tau(f_2)(x,y,t)=-2b_1y+4b_2+4b_3(x-yt)+4b_4y+12b_5t+12b_6yt$$
and hence $\tau^2(f_2)=0$.
This gives a complex 6-dimensional family of proper biharmonic functions on $\SL2$.
\end{example}

\section{Product spaces}

The last two remaining model geometries on Thurston's list are the product spaces
$\H^2\times\rn$ and $\s^2\times\rn$.  Before dealing with these we develop some
general theory for product spaces. In this situation we can separate variables.

We assume that $(M,g)=(M_1,g_1)\times(M_2,g_2)$ is the
product of two Riemannian manifolds. Further that $f_1:M_1\to\cn$, $f_2:M_2\to\cn$
are complex-valued functions and $f:M\to\cn$ is given by the product
$$f(x,y)=f_1(x)\cdot f_2(y).$$

The following result is a wide going generalisation, of Lemma 2.4 of the paper
\cite{Ou-1} by Ye-Lin Ou, in this special case.

\begin{lemma}\label{lemma-product}
Let $(M,g)=(M_1,g_1)\times(M_2,g_2)$ be the product of two Riemannian manifolds.
Further let $f_1:M_1\to\cn$, $f_2:M_2\to\cn$ be complex-valued functions
and $f:M\to\cn$ be given by $f(x,y)=f_1(x)\cdot f_2(y)$.
Then the $n$-th tension field satisfies
$$\tau^n(f)=\sum_{k=0}^n \binom nk\tau^{n-k}(f_1)\cdot\tau^k(f_2).$$
\end{lemma}

\begin{proof}
It follows directly from the constructions of the manifold $(M,g)$
and the function $f=f_1\cdot f_2$ that the conformality operator $\kappa$ satisfies
$$
\kappa(f_1,f_2)=\kappa(\tau(f_1),f_2)=\kappa(f_1,\tau(f_2))=0.
$$
For the tension fields $\tau(f)$ and $\tau^2(f)$ we then have
\begin{eqnarray*}
\tau(f)&=&\tau(f_1)\cdot f_2+2\,\kappa(f_1,f_2)+f_1\cdot\tau(f_2)\\
&=&\tau(f_1)\cdot f_2+f_1\cdot\tau(f_2)
\end{eqnarray*}
and
\begin{eqnarray*}
\tau^2(f)&=&\tau(\tau(f_1)\cdot f_2)+\tau(f_1\cdot\tau(f_2))\\
&=&\tau^2(f_1)\cdot f_2+2\,\kappa(\tau(f_1),f_2)+\tau(f_1)\cdot\tau(f_2)\\
& &\quad +\tau(f_1)\cdot\tau(f_2)+2\,\kappa(f_1,\tau(f_2))+f_1\cdot\tau^2(f_2)\\
&=&\tau^2(f_1)\cdot f_2+2\,\tau(f_1)\cdot\tau(f_2)+f_1\cdot\tau^2(f_2).
\end{eqnarray*}
The rest follows by induction.
\end{proof}

The general formula of Lemma \ref{lemma-product} has the following immediate consequence.

\begin{proposition}\label{proposition-product}
Let $(M,g)=(M_1,g_1)\times(M_2,g_2)$ be the product of two Riemannian manifolds.
Further let $f_1:M_1\to\cn$, $f_2:M_2\to\cn$ be complex-valued functions
and $f:M\to\cn$ be given by $f(x,y)=f_1(x)\cdot f_2(y)$.
\begin{itemize}
\item[i.] If $f_1$ is proper harmonic and $f_2$ is proper biharmonic
then their product $f$ is proper biharmonic.
\item[ii.] If $f_1$ and $f_2$ are proper biharmonic then their product $f$ is proper triharmonic.
\end{itemize}
\end{proposition}

\begin{proof}
The result follows directly from the following consequences of Lemma \ref{lemma-product}
$$\tau(f)=\tau(f_1)\cdot f_2+f_1\cdot\tau(f_2),$$
$$\tau^2(f)=\tau^2(f_1)\cdot f_2+2\,\tau(f_1)\cdot\tau(f_2)+f_1\cdot\tau^2(f_2).$$
$$\tau^3(f)=\tau^3(f_1)\cdot f_2+3\,\tau^2(f_1)\cdot\tau(f_2)+3\,\tau(f_1)\cdot\tau^2(f_2)+f_1\cdot\tau^3(f_2).$$
\end{proof}

\begin{remark}
It should be noted that if the functions $f_1$ and $f_2$ in Proposition
\ref{proposition-product} are both proper biharmonic then it follows from the above calculations that
$$\tau^2(f)=2\, \tau(f_1)\cdot\tau(f_2)\neq 0.$$
This means that the product is not biharmonic.  This contradicts
the stated result in Proposition 2.1 of the paper \cite{Ou-2}.
\end{remark}

The first part of Proposition \ref{proposition-product} can now easily be generalised to the following.

\begin{proposition}\label{proposition-product-general}
Let $(M,g)=(M_1,g_1)\times(M_2,g_2)$ be the product of two Riemannian manifolds.
Further let $f_1:M_1\to\cn$, $f_2:M_2\to\cn$ be complex-valued functions
and $f:M\to\cn$ be given by $f(x,y)=f_1(x)\cdot f_2(y)$.
If $f_1$ is proper harmonic and $f_2$ is proper r-harmonic
then their product $f$ is proper r-harmonic.
\end{proposition}

\begin{proof}
It immediately follows from Proposition \ref{proposition-product} that
$$\tau(f)=f_1\cdot\tau(f_2)\neq 0,\dots , \tau^{r-1}(f)=f_1\cdot\tau^{r-1}(f_2)\neq 0\ \ \text{and}\ \ \tau^r(f)=0.$$
\end{proof}

\section{The product spaces $\H^2\times\rn$ and $\s^2\times\rn$}

Let us now consider the hyperbolic disc $H^2$ with its standard Riemannian metric of constant curvature $-1$.  We then equip the product space $\H^2\times\rn$ with its product metric.  In the standard global coordinates $(z,t)$ on $\H^2\times\rn$ the operators $\tau$ and $\kappa$ are then given by
\begin{equation*}\label{tension-H2-R}
\tau(f)=4(1-z\bar z)^2\frac {\partial ^2 f}{\partial z\partial \bar z}+\frac {\partial ^2 f}{\partial t^2}
\end{equation*}
and
\begin{equation*}\label{conformality-H2-R}
\kappa(f,h)=2(1-z\bar z)^2(
\frac{\partial f}{\partial z}\frac{\partial h}{\partial \bar z}
+\frac{\partial f}{\partial \bar z}\frac{\partial h}{\partial z})
+\frac{\partial f}{\partial t}\frac{\partial h}{\partial t}.
\end{equation*}

As a direct consequence of Proposition \ref{proposition-product} we have the following result.

\begin{corollary}
Let the functions $f,g:\H^2\to\cn$ be holomorphic and $p:\rn\to\rn$ be a polynomial
$$p(t)=b_0+b_1t+b_2t^2+b_3t^3$$ such that $(b_2,b_3)\neq 0$.  Then the function
$$F(z,t)=(f(z)+g(\bar z))\cdot p(t)$$
is proper biharmonic on the product space $\H^2\times\rn$.
\end{corollary}

\begin{proof}
It is well-known that a function $f_1:\H^2\to\cn$ is harmonic if and only if it is the sum of a holomorphic function and an  anti-holomorphic one.  Further any proper biharmonic function $p:\rn\to\rn$ is a polynomial $p(t)=b_0+b_1t+b_2t^2+b_3t^3$ with $(b_2,b_3)\neq 0$.  The result is now an immediate consequence of Proposition \ref{proposition-product}.
\end{proof}

As a further consequence of Proposition \ref{proposition-product} we have the following result. The proper biharmonic functions on the hyperbolic disc are described in Appendix \ref{appendix}.

\begin{corollary}
Let the function $f:\H^2\to\cn$ be proper biharmonic on the hyperbolic disc and $p:\rn\to\rn$ be a non-zero harmonic polynomial $p(t)=a_0+a_1t$. Then the function
$$
F(z,\bar z, t)=f(z,\bar z)\cdot p(t)
$$ is proper biharmonic on the product space $\H^2\times\rn$.
\end{corollary}

The story for the product space $\s^2\times\rn$ is much the same as that of $\H^2\times\rn$.  Because of the maximum principle for harmonic functions we need to consider the punctured sphere $P=\s^2\setminus\{p\}$ instead of $\s^2$.  This we model as the complex plane equipped with its historic Riemannian metric
$$ds^2=\frac 4{(1+(x^2+y^2))^2}(dx^2+dy^2)$$
of constant curvature +1. In this case Proposition \ref{proposition-product} leads to the following constructions.

\begin{corollary}
Let the functions $f,g:P\to\cn$ be holomorphic on the punctured sphere.
Further let $p:\rn\to\rn$ be a polynomial
$$
p(t)=b_0+b_1t+b_2t^2+b_3t^3,
$$
such that $(b_2,b_3)\neq 0$.  Then the function
$$
F(z,\bar z,t)=(f(z)+g(\bar z))\cdot p(t)
$$ is proper biharmonic on the product space $P\times\rn$.
\end{corollary}

The proper biharmonic functions on the punctured sphere $P$ are described in Appendix \ref{appendix}.

\begin{corollary}
Let the function $f:P\to\cn$ be proper biharmonic on the punctured sphere and
$p:\rn\to\rn$ be a non-zero harmonic polynomial $p(t)=a_0+a_1t$. Then the function
$$
F(z,\bar z,t)=f(z,\bar z)\cdot p(t)
$$
is proper biharmonic on the product space $P\times\rn$.
\end{corollary}

\section{Acknowledgements}
The author is grateful to Ye-Lin Ou, Stefano Montaldo and Andrea Ratto for useful discussions related to this work.

\appendix

\section{Biharmonic functions on $\H^2$ and $\s^2$}\label{appendix}

For the completeness of this paper we here state a few well-know facts.  Let
$$
\H^2=\{z=(x+iy)\in\cn|\ z\bar z=(x^2+y^2)<1\}
$$
be the standard hyperbolic disc of constant curvature $-1$.  Its Riemannian metric
$$
ds^2=\frac 4{(1-(x^2+y^2))^2}(dx^2+dy^2)
$$
is conformally equivalent to the flat Euclidean metric of the open unit disc.
The corresponding Laplace-Beltrami operator is given by
$$
\tau(f)(z,\bar z)=4{(1-(z\bar z))^2}\frac{\partial^2f}{\partial z\partial\bar z}(z,\bar z).
$$
A harmonic function $h:\H^2\to\cn$ is the sum of a holomorphic and an anti-holomorphic one.
This means that a biharmonic function $F:\H^2\to\cn$ must satisfy
$$
4{(1-(z\bar z))^2}\frac{\partial^2F}{\partial z\partial\bar z}(z,\bar z)=h(z,\bar z)=h_1(z)+h_2(\bar z),
$$
where $h_1,h_2:\H^2\to\cn$ are holomorphic.  By dividing and integrating locally we get
$$
F(z,\bar z)=\frac 14\iint \frac{h(z,\bar z)}{(1-(z\bar z))^2}dzd\bar z.
$$

\begin{example}
Let $h:\H^2\to\cn$ be the constant harmonic function $4$.  Then a double integration gives
$$
F_1(z,\bar z)=\iint \frac 1{(1-(z\bar z))^2}d\bar zdz=-\log(1-z\bar z)+H_1(z,\bar z)
$$
or
$$
F_2(z,\bar z)=\iint \frac 1{(1-(z\bar z))^2}dzd\bar z=-\log(1-z\bar z)+H_2(z,\bar z),
$$
depending on the order of integration.  Here $H_1$ and $H_2$ are local harmonic functions.
This means that, modulo a harmonic function, we obtain the well-known globally defined
proper biharmonic function
$$
F:\H^2\to\cn\ \ \ \text{with}\ \ F(z,\bar z)=-\log(1-z\bar z).
$$
\end{example}

\begin{remark}
Let $\s^2$ be the standard round sphere of constant curvature $+1$.  Then the maximum principle tells us that every globally defined harmonic function on $\s^2$ is constant.  For this reason we will instead consider the punctured sphere $P=\s^2\setminus\{p\}$.  We model this as the complex plane $\cn$ with the conformally flat Riemannian metric
$$
ds^2=\frac 4{(1+(x^2+y^2))^2}(dx^2+dy^2).
$$
In this case we can proceed in exactly the same manner as we did for the hyperbolic disc.  For example, we yield the well-known proper biharmonic function
$$
F:P\to\cn\ \ \text{with}\ \ F(z,\bar z)=\log(1+z\bar z).
$$
\end{remark}


\begin{thebibliography}{99}

\bibitem{BW-book}
P. Baird, J.C. Wood,
{\it Harmonic morphisms between Riemannian manifolds},
The London Mathematical Society Monographs {\bf 29},
Oxford University Press (2003).

\bibitem{Gud-Mon-Rat-1}
S. Gudmundsson, S. Montaldo, A. Ratto,
{\it Biharmonic functions on the classical compact simple Lie groups},
J. Geom. Anal. (to appear).

\bibitem{Gud-13}
S. Gudmundsson,
{\it Biharmonic functions on the special unitary group SU(2)},
Differential Geom. Appl. {\bf 53} (2017), 137-147.

\bibitem{Ou-1}
Y.-L. Ou,
{\it Biharmonic morphisms between Riemannian manifolds},
in "Geometry and Topology of Submanifolds X" (Beijing/Berlin, 1999),
World Sci. Publ. (2000), 231-239.

\bibitem{Ou-2}
Y.-L. Ou,
{\it Some constructions of biharmonic maps and Chen's conjecture on biharmonic hypersurfaces},
J. Geom. Phys. {\bf 62} (2012), 751-762.

\bibitem{Sco}
P. Scott,
{\it The geometries of 3-manifolds},
Bull. London Math. Soc. {\bf 15}, (1983), 401-487.

\bibitem{Thu}
W. P. Thurston, William P.
{\it Three-dimensional geometry and topology},
Princeton Mathematical Series {\bf 35}.
Princeton University Press (1997).



\end{thebibliography}
\end{document}